\documentclass{amsart}

\usepackage{amsmath,amsfonts,amssymb}
\usepackage{color}
\usepackage{tikz}
\usepackage{ytableau}
\usepackage{float}
\usetikzlibrary{arrows.meta,chains,matrix,decorations.pathreplacing, arrows}
\usetikzlibrary{patterns}
\usepackage{mathabx}
\usepackage{comment}

\newtheorem*{theorem*}{Theorem}

\newtheorem{theorem}{Theorem}[section]
\newtheorem{lemma}[theorem]{Lemma}
\newtheorem{proposition}[theorem]{Proposition}
\newtheorem{corollary}[theorem]{Corollary}

\theoremstyle{definition}
\newtheorem{definition}[theorem]{Definition}

\theoremstyle{remark}

\theoremstyle{observation}

\numberwithin{equation}{section}

\newlength\cellsize 

\savebox2{%
\begin{picture}(10,10)
\put(0,0){\line(1,0){10}}
\put(0,0){\line(0,1){10}}
\put(10,0){\line(0,1){10}}
\put(0,10){\line(1,0){10}}
\end{picture}}

\newcommand\cellify[1]{\def\thearg{#1}\def\nothing{}%
\ifx\thearg\nothing\vrule width0pt height\cellsize depth0pt%
  \else\hbox to 0pt{\usebox2\hss}\fi%
  \vbox to 10\unitlength{\vss\hbox to 10\unitlength{\hss$#1$\hss}\vss}}

\newcommand\tableau[1]{\vtop{\let\\=\cr
\setlength\baselineskip{-10000pt}
\setlength\lineskiplimit{10000pt}
\setlength\lineskip{0pt}
\halign{&\cellify{##}\cr#1\crcr}}}

\savebox7{
\begin{picture}(10,10)
  \put(5,5){\circle{9.4}}
\end{picture}}

\newcommand{\cirfy}[1]{\def\thearg{#1}\def\nothing{}%
\ifx\thearg\nothing\vrule width0pt height\cellsize depth0pt%
  \else\hbox to 0pt{\usebox7\hss}\fi%
  \vbox to 10\unitlength{\vss\hbox to 10\unitlength{\hss$#1$\hss}\vss}}

\newcommand\cirtab[1]{\vtop{\let\\=\cr
\setlength\baselineskip{-10000pt}
\setlength\lineskiplimit{10000pt}
\setlength\lineskip{0pt}
\halign{&\cirfy{##}\cr#1\crcr}}}

\input xy
\usepackage[all]{xy}
\xyoption{line}
\xyoption{arrow}
\xyoption{color}
\SelectTips{cm}{}

\usepackage{tikz}
\usetikzlibrary{decorations.markings}
\usetikzlibrary{decorations.pathreplacing}
\newcommand{\hackcenter}[1]{
 \xy (0,0)*{#1}; \endxy}

\tikzstyle directed=[postaction={decorate,decoration={markings,
    mark=at position #1 with {\arrow{>}}}}]
\tikzstyle rdirected=[postaction={decorate,decoration={markings,
    mark=at position #1 with {\arrow{<}}}}]

\usetikzlibrary{calc}
\usepackage{relsize}

\tikzset{fontscale/.style = {font=\relsize{#1}}
    }

\newcommand{\newword}[1]{\textbf{\emph{#1}}}

\newcommand{\g}{\ensuremath\mathfrak{g}}
\newcommand{\fsl}{\ensuremath\mathfrak{sl}}
\newcommand{\asl}{\ensuremath\widehat{\mathfrak{sl}}}

\newcommand{\B}{\mathcal{B}}
\newcommand{\aB}{\widetilde{\mathcal{B}}}
\newcommand{\wt}{\ensuremath\mathrm{wt}}
\newcommand{\D}{\mathcal{E}}

\newcommand{\key}{\ensuremath\kappa}
\newcommand{\SSKD}{\ensuremath\mathrm{SSKD}}

\newcommand{\maj}{\ensuremath\mathrm{maj}}

\title[Affine Demazure crystals]{Affine Demazure crystals for specialized nonsymmetric Macdonald polynomials}  

\author[S. Assaf]{Sami Assaf}
\address{Department of Mathematics, University of Southern California, 3620 S. Vermont Ave., Los Angeles, CA 90089-2532, U.S.A.}
\email{shassaf@usc.edu}
\thanks{S.A. supported in part by NSF DMS-1763336. S.K. supported in part by NSF DMS-1802328.}

\author[N. Gonz\'{a}lez]{Nicolle Gonz\'{a}lez}
\address{UCLA Department of Mathematics, Los Angeles, CA 90095-1555, U.S.A.}
\email{nicolle@math.ucla.edu}

\thanks{S.A. supported in part by NSF DMS-1763336.}

\keywords{Affine Demazure crystal, affine Demazure character, nonsymmetric Macdonald polynomial}

\begin{document}

\begin{abstract}
  We give a crystal-theoretic proof that nonsymmetric Macdonald polynomials specialized to $t=0$ are affine Demazure characters. We explicitly construct an affine Demazure crystal on semistandard key tabloids such that removing the affine edges recovers the finite Demazure crystals constructed earlier by the authors. We also realize the filtration on highest weight modules by Demazure modules by defining explicit embedding operators which, at the level of characters, parallels the recursion operators of Knop and Sahi for specialized nonsymmetric Macdonald polynomials. Thus we prove combinatorially in type A that every affine Demazure module admits a finite Demazure flag.
\end{abstract}
\maketitle
%
\section{Introduction}
%
\label{sec:introduction}

Macdonald \cite{Mac88} defined an important family of polynomials that forms a basis of symmetric polynomials in $\mathbb{C}(q,t)[x_1,\ldots,x_n]$. Opdam \cite{Opd95} defined a nonsymmetric generalization of these polynomials, also discovered by Macdonald \cite{Mac96}, that form a basis of $\mathbb{C}(q,t)[x_1,\ldots,x_n]$. The symmetric Macdonald polynomials have deep connections to representation theory and geometry, and many new connections have been discovered by studying their nonsymmetric generalizations.

Sanderson \cite{San00} made the connection between nonsymmetric Macdonald polynomials and affine Demazure modules \cite{Dem74a} through the character formula stated by Demazure \cite{Dem74} and proved rigorously by Andersen \cite{And85}. Using the recurrence formula for nonsymmetric Macdonald polynomials discovered independently by Knop \cite{Kno97} and Sahi \cite{Sah96}, Sanderson proved the nonsymmetric Macdonald polynomials specialized at $t=0$ are affine Demazure characters. One of our main results is a crystal theoretic lift of Knop and Sahi's operators (specialized at $t=0$) as generators of affine Demazure crystals.

Haglund, Haiman and Loehr \cite{HHL08} gave a combinatorial formula for nonsymmetric Macdonald polynomials, also based on the recurrence of Knop and Sahi. Assaf \cite{Ass18} used this formula and weak dual equivalence \cite{Ass-W} to prove the nonsymmetric Macdonald polynomial specialized at $t=0$ is a nonnegative $q$-graded sum of finite Demazure characters. In this paper, we connect the affine and finite results for nonsymmetric Macdonald polynomials through crystals, bridging the combinatorics with the representation theoretic perspective of Sanderson. 

Kashiwara \cite{Kas91} combinatorialized certain highest weight modules through his study of crystals which he generalized to Demazure modules with Demazure crystals \cite{Kas93}. Demazure modules form a filtration of highest weight modules compatible with the Bruhat order of the corresponding Weyl group. This filtration descends to the appropriate crystals via Demazure operators and is realized on the characters by the recursion operators of Knop and Sahi. Kashiwara and Nakashima and, independently, Littelmann gave explicit tableaux models for finite type crystals \cite{KN94,Lit95}, and in a similar spirit Assaf and Schilling gave an explicit tableaux model for Demazure crystals in type A \cite{ASc18}. Assaf and Gonz\'{a}lez \cite{AG21} recently gave a crystal theoretic proof of Assaf's result \cite{Ass18} decomposing nonsymmetric Macdonald polynomials specialized at $t=0$ as a nonnegative $q$-graded sum of finite Demazure characters, leading to more explicit formulas.

In this paper, we give an explicit construction of affine Demazure crystals on semistandard key tabloids, the combinatorial objects that generate nonsymmetric Macdonald polynomials. This gives a new combinatorial proof of Sanderson's result that the nonsymmetric Macdonald polynomial specialized at $t=0$ is the affine Demazure character. Generalizing our earlier finite crystal construction \cite{AG21}, we define affine edges to the finite Demazure crystal on semistandard key tabloids. We provide a realization of the Bruhat filtration on Demazure crystals via embedding operators, which correspond to a combinatorial analogue of the Demazure operators and recover Knop and Sahi's operators (at $t=0$) at the level of characters. As a corollary, we give explicit formulas for affine Demazure characters as nonnegative graded sums of finite Demazure characters. 

Our combinatorial results are type A specific.  Cherednik \cite{Che95} uniformly generalized nonsymmetric Macdonald polynomials to all types. Similarly,  Kumar \cite{Kum87} gave general type formulas for Demazure characters. Inspired by Sanderson, Ion \cite{Ion03} proved a general connection between nonsymmetric Macdonald polynomials and affine Demazure characters.  Kumar conjectures our type A result holds in greater generality,  in particular, that every affine Demazure module admits a finite Demazure flag.

%
\section{Demazure crystals}
%
\label{sec:crystals}

Given a complex, semi-simple Lie algebra $\g$ with dominant integral weights $P^{+}\subset P$, there is a unique irreducible highest weight $\g$-module $V^{\lambda}$ for each $\lambda\in P^{+}$. For $W$ the Weyl group of $\g$, each weight space $V^{\lambda}_w$ of weight $w \cdot \lambda$ is one-dimensional. The \newword{Demazure modules} $D_{w \cdot \lambda}$ are the $\mathfrak{b}$-submodules generated by the action of a Borel subalgebra $\mathfrak{b}$ on $V^{\lambda}_w$. The Demazure modules form a filtration of $V^{\lambda}$ compatible with Bruhat order on $W$, so that $w \leq w'$ if and only if $D_{w \cdot \lambda} \subseteq D_{w' \cdot \lambda}$. Kashiwara proved each irreducible module $V^{\lambda}$ and Demazure module $D_{w \cdot \lambda}$ has a crystal basis encoding important combinatorial data including its character. 

A \newword{crystal basis} for a $\g$ consists of a nonempty set $\B$ not containing $0$ and \newword{crystal operators} $e_i, f_i  :  \B \rightarrow \B \cup \{0\}$ such that $e_i(b)=b^{\prime}$ if and only if $f_i(b^{\prime}) = b$ for $b,b^{\prime}\in\B$ satisfying certain conditions. These conditions can be described in terms of the three maps $\wt,\varepsilon,\varphi : \B \rightarrow P$, called the weight map and string lengths.

For $\g = \fsl_n$ or $\asl_n$, the integral weights are weak compositions $\alpha$ of length $n$, and the \newword{degree} of $\alpha\in P$ is $\alpha_1+\alpha_2+\cdots+\alpha_n$. The \newword{weight map} satisfies $\wt(b^{\prime}) = \wt(b) + (\mathbf{e}_i-\mathbf{e}_{i+1})$ whenever $e_i(b)=b^{\prime}$ for $1 \leq i < n$ (here $\mathbf{e}_k$ is the composition with $1$ in position $k$ and all other entries $0$). For $\g = \asl_n$, in addition we have $\wt(b^{\prime}) = \wt(b) + (\mathbf{e}_{n}-\mathbf{e}_{1})$ whenever $e_0(b)=b^{\prime}$. The \newword{string lengths} satisfy $\varphi_i(b) - \varepsilon_i(b) = \wt(b)_i - \wt(b)_{i+1}$ for $i>0$ and $\varphi_0(b) - \varepsilon_0(b) = \wt(b)_n - \wt(b)_{1}$, and are given explicitly by 
\begin{eqnarray*}
  \varepsilon_i(b) = \max\{j \geq 0 \mid e_i^j(b) \neq 0 \}
  & \text{and} &
  \varphi_i(b) = \max\{j \geq 0 \mid f_i^j(b) \neq 0 \}.
\end{eqnarray*}
Since the crystal basis is a \emph{basis}, the \newword{dimension} of a crystal is the size of the set $\B$. We often abuse notation by referring to the \newword{crystal data} $(\B,\wt,\{e_i\},\{f_i\})$ simply by the set $\B$ when the weight map and crystal operators are already defined. 

\begin{figure}[ht]
  \begin{displaymath}
    \begin{tikzpicture}[xscale=1.5,yscale=1]
      \node at (0,0)   (a) {$\tableau{1}$};
      \node at (1,0)   (b) {$\tableau{2}$};
      \node at (2,0)   (c) {$\tableau{3}$};
      \node at (3,0)   (d) {$\cdots$};
      \node at (4,0)   (e) {$\tableau{n}$};
      \draw[thick,color=blue  ,->] (a) -- (b) node[midway,above] {$\scriptstyle 1$} ;
      \draw[thick,color=purple,->] (b) -- (c) node[midway,above] {$\scriptstyle 2$} ;
      \draw[thick,color=red   ,->] (c) -- (d) node[midway,above] {$\scriptstyle 3$} ;
      \draw[thick,color=orange,->] (d) -- (e) node[midway,above] {$\scriptstyle n-1$} ;
      \draw[thick,color=violet,->] (e.south) to[out=200,in=340] node[above] {$\scriptstyle 0$} (a.south)  ;
    \end{tikzpicture}
  \end{displaymath}
  \caption{\label{fig:standard}The standard crystal $\aB(n)$ for $\asl_n$; removing the $0$-edge gives the standard crystal $\B(n)$ for ${\fsl}_n$.}
\end{figure}
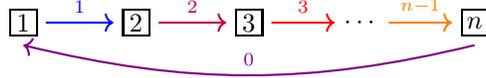

The \newword{standard crystal} $\B(n)$ for $\fsl_n$ ($\aB(n)$ for $\asl_n$) has crystal basis $\left\{\raisebox{-0.3\cellsize}{$\tableau{i}$}\mid 1 \leq i < n \right\}$, weight map $\wt\left(\,\raisebox{-0.3\cellsize}{$\tableau{i}$}\,\right) = \mathbf{e}_i$, and finite (and affine) crystal operators as shown in Fig.~\ref{fig:standard}, where we draw a directed $i$-edge from $b'$ to $b$ if $f_i(b')=b$.

Given two crystals $\B_1$ and $\B_2$, the \newword{tensor product} $\B_1 \otimes \B_2$ is the set $\B_1 \otimes \B_2$ with $\wt(b_1 \otimes b_2) = \wt(b_1) + \wt(b_2)$ and crystal operators $e_i, f_i$ defined by
\begin{equation}
  f_i(b_1 \otimes b_2) = \left\{ \begin{array}{rl}
    f_i(b_1) \otimes b_2 & \mbox{if } \varepsilon_i(b_2) < \varphi_i(b_1), \\
    b_1 \otimes f_i(b_2) & \mbox{if } \varepsilon_i(b_2) \geq \varphi_i(b_1).
  \end{array} \right.
\end{equation}

For example, Fig.~\ref{fig:tensor} shows the tensor product of two copies of the standard crystal $\aB(3)$ which, as an $\asl_n$-crystal, is connected. However, as an $\fsl_n$-crystal, we ignore the $0$ edges resulting in two connected components, one of dimension $6$ with highest weight $(2,0,0)$ and the other of dimension $3$ with highest weight $(1,1,0)$. 

\begin{figure}[ht]
  \begin{center}
    \begin{tikzpicture}[xscale=2,yscale=1.2]
      \node at (1,2)   (U11)  {$\cellify{1}\otimes\cellify{1}$};
      \node at (2,2)   (U21)  {$\cellify{2}\otimes\cellify{1}$};
      \node at (3,2)   (U31)  {$\cellify{3}\otimes\cellify{1}$};
      \node at (1,1)   (U12)  {$\cellify{1}\otimes\cellify{2}$};
      \node at (2,1)   (U22)  {$\cellify{2}\otimes\cellify{2}$};
      \node at (3,1)   (U32)  {$\cellify{3}\otimes\cellify{2}$};
      \node at (1,0)   (U13)  {$\cellify{1}\otimes\cellify{3}$};
      \node at (2,0)   (U23)  {$\cellify{2}\otimes\cellify{3}$};
      \node at (3,0)   (U33)  {$\cellify{3}\otimes\cellify{3}$};
      \draw[thick,blue  ,->](U11) -- (U21) node[midway,above]{$\scriptstyle 1$};
      \draw[thick,blue  ,->](U21) -- (U22) node[midway,left] {$\scriptstyle 1$};
      \draw[thick,blue  ,->](U31) -- (U32) node[midway,left] {$\scriptstyle 1$};
      \draw[thick,blue  ,->](U13) -- (U23) node[midway,above]{$\scriptstyle 1$};
      \draw[thick,purple,->](U21) -- (U31) node[midway,above]{$\scriptstyle 2$};
      \draw[thick,purple,->](U22) -- (U32) node[midway,above]{$\scriptstyle 2$};
      \draw[thick,purple,->](U32) -- (U33) node[midway,left] {$\scriptstyle 2$};
      \draw[thick,purple,->](U12) -- (U13) node[midway,left] {$\scriptstyle 2$};
      \draw[thick,violet,->](U32.south) to[out=200,in=340] node[above] {$\scriptstyle 0$} (U12.south) ;
      \draw[thick,violet,->](U33.south) to[out=200,in=340] node[above] {$\scriptstyle 0$} (U13.south) ;
      \draw[thick,violet,->](U13.west) to[out=100,in=260] node[left]   {$\scriptstyle 0$} (U11.west) ;
      \draw[thick,violet,->](U23.west) to[out=100,in=260] node[left]   {$\scriptstyle 0$} (U21.west) ;
    \end{tikzpicture}
    \caption{\label{fig:tensor}Tensor product of two standard crystals $\aB(3)$, giving the degree $2$ affine crystal for $\asl_3$.}
  \end{center}
\end{figure}
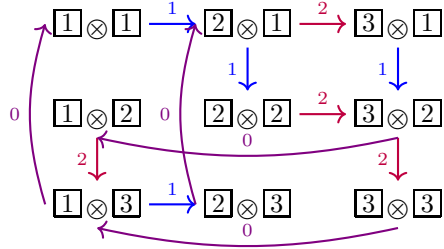

For $\g = \fsl_n$ or $\asl_n$, the Weyl group $W$ for $\g$ acts on $\alpha \in P$ by
\begin{eqnarray} \label{eq:W-action}
  s_i \cdot (\alpha_1,\ldots,\alpha_n) & = & (\alpha_1,\ldots,\alpha_{i-1}, \alpha_{i+1}, \alpha_i, \alpha_{i+2},\ldots,\alpha_n) , \\
  s_0 \cdot (\alpha_1,\ldots,\alpha_n) & = & (\alpha_n+1,\alpha_2,\ldots,\alpha_{n-1},\alpha_1-1) ,
\end{eqnarray}
where the $s_i$ are the simple reflections that generate $W$. We extend Bruhat order to $P$ by writing each $\alpha\in P$ as $w \cdot \lambda$ for a unique minimum length $w\in W$ and unique $\lambda$, where in the finite case $\lambda\in P^+$ and in the affine case, $\lambda$ of degree $k$ is
\begin{equation}\label{eq:eta}
  \eta_{n,k} = (\overbrace{m+1,\ldots,m+1}^r,\overbrace{m,\ldots,m}^{n-r}),
\end{equation}
where $k$ can be expressed uniquely as $k = m n + r$ with $m\geq 0$ and $0 \leq r < n$. 

A \newword{highest weight element} is any $u\in\B$ such that $e_i(u)=0$ for all $1 \leq i < n$. From the relations between the crystal operators and the weight map and string lengths, it follows that each highest weight element $u$ satisfies $\wt(u) \in P^{+}$ and each connected component of a finite crystal has a unique highest weight element. 

Let $\lambda\in P^{+}$ be a dominant integral weight of degree $k$. For $\g = \fsl_n$, the crystal basis $\B(\lambda)$ for the irreducible module $V^{\lambda}$ is any connected component of $\B(n)^{\otimes k}$ with highest weight $\lambda$. For $\g = \asl_n$, the finite-dimensional crystal basis $\aB(\lambda)$ for the finite-dimensional $V^{\lambda}$ is the connected crystal $\aB(n)^{\otimes k}$. For example, Fig.~\ref{fig:tensor} shows the affine crystal $\aB(2,0,0)$ for $\asl_3$ and, ignoring the $0$ edges, the two finite crystals $\B(2,0,0)$ and $\B(1,1,0)$ for $\fsl_3$.

To realize the crystals for the Demazure modules, we consider the \newword{Demazure operators} $\mathfrak{D}_i$ defined on any subset $X$ of a crystal $\B$ by
\begin{equation}
  \mathfrak{D}_i X = \{ b \in \B \mid e_i^j(b) \in X \mbox{ for some } j \geq 0 \} .
  \label{e:D}
\end{equation}
These operators satisfy the Coxeter relations for $W$, and so we set $\mathfrak{D}_w = \mathfrak{D}_{i_1} \cdots \mathfrak{D}_{i_{\ell}}$ for any reduced expression $s_{i_1} \cdots s_{i_{\ell}}$ for $w\in W$.

For $\lambda\in P^{+}$ and $w\in W$, the \newword{(affine) Demazure crystal} $\B_w(\lambda)$ or $\aB_w(\lambda)$ for the Demazure module $D_{w \cdot \lambda}$ is 
\begin{equation}
  \B_w(\lambda) = \mathfrak{D}_{w} \{ u_\lambda \} \qquad
   \text{or} \qquad
  \aB_w(\lambda) = \mathfrak{D}_{w} \{ \tilde{u}_\lambda \}
\end{equation}
where $u_{\lambda}$ is the unique highest weight element in $\B(\lambda)$ of weight $\lambda$, and $\tilde{u}_{\lambda}$ is the highest weight element
\begin{equation}\label{eq:u-tilde}
  \left(\cellify{1} \otimes \cdots \otimes \cellify{n}\right)^{\otimes m} \otimes \cellify{1} \otimes \cdots \otimes \cellify{r}
\end{equation}
of $\aB(\lambda)$, where $\lambda$ has degree $k = mn+r$ with $m\geq 0$ and $0 \leq r < n$.

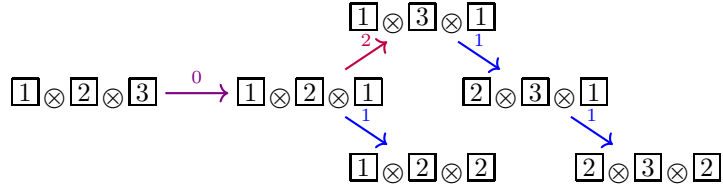
\begin{figure}[ht]
  \begin{displaymath}
    \begin{tikzpicture}[xscale=1.5,yscale=1]
      \node at (-1,1)  (A) {$\cellify{1}\otimes\cellify{2}\otimes\cellify{3}$};
      \node at (1,1)   (B) {$\cellify{1}\otimes\cellify{2}\otimes\cellify{1}$};
      \node at (2,2)   (C) {$\cellify{1}\otimes\cellify{3}\otimes\cellify{1}$};
      \node at (2,0)   (D) {$\cellify{1}\otimes\cellify{2}\otimes\cellify{2}$};
      \node at (3,1)   (E) {$\cellify{2}\otimes\cellify{3}\otimes\cellify{1}$};
      \node at (4,0)   (F) {$\cellify{2}\otimes\cellify{3}\otimes\cellify{2}$};
      \draw[thick,color=violet,->] (A) -- (B) node[midway,above] {$\scriptstyle 0$} ;
      \draw[thick,color=blue  ,->] (B) -- (D) node[midway,above] {$\scriptstyle 1$} ;
      \draw[thick,color=purple,->] (B) -- (C) node[midway,above] {$\scriptstyle 2$} ;
      \draw[thick,color=blue  ,->] (C) -- (E) node[midway,above] {$\scriptstyle 1$} ;
      \draw[thick,color=blue  ,->] (E) -- (F) node[midway,above] {$\scriptstyle 1$} ;
    \end{tikzpicture}
  \end{displaymath}
  \caption{\label{fig:dem-crystal}The affine Demazure crystal $\aB_{s_1 s_2 s_0}(1,1,1)$.}
\end{figure}

For example, Fig.~\ref{fig:dem-crystal} shows the affine Demazure crystal $\aB_w(\lambda)$ for $\asl_3$ with $\lambda=(1,1,1)$ of degree $3$ and $w = s_1 s_2 s_0$, where $\tilde{u}_{\lambda}$ is the leftmost element.

The Demazure crystals $\B_w(\lambda)$ and $\aB_w(\lambda)$ form a filtration of the highest weight crystal $\B(\lambda)$ and $\aB(\lambda)$ compatible with Bruhat order in exactly the same way Demazure modules form a filtration of irreducible modules. One of our main results is to give an explicit construction of affine Demazure crystals on tableaux-like objects that avoids the construction of the tensor product. Moreover, our direct construction circumvents the iterative use of Demazure operators, but we nonetheless give nested crystal embeddings that realize the filtration under the Bruhat order.

%
\section{Tabloid crystals}
%
\label{sec:tabloids}

For $\g = \fsl_n$, the character of the irreducible module $V^{\lambda}$ is the \newword{Schur polynomial} $s_{\lambda}(x_1,\ldots,x_n)$, defined combinatorially as the generating polynomial of semi-standard Young tableaux. Thus tableaux are the natural choice for the underlying set of the crystal basis $\B(\lambda)$ as realized by Kashiwara and Nakashima \cite{KN94} and Littelmann \cite{Lit95}. More generally, the \newword{Demazure character} $\key_{w \cdot \lambda}(x_1,\ldots,x_n)$ of the Demazure module $V_w^{\lambda}$ can be defined combinatorially as the generating polynomial of semi-standard key tableaux, making this a natural choice for the underlying set of the Demazure crystal basis $\B_w(\lambda)$ as realized by Assaf and Schilling \cite{ASc18}.

For $\g = \asl_n$, Sanderson \cite{San00} proved the graded character of the module $V^{\lambda}$ is (up to rescaling) the specialized symmetric Macdonald polynomial $P_{\lambda}(x_1,\ldots,x_n;q,0)$, using notation from \cite{Mac88}. The character of the Demazure module $V_w^{\lambda}$ is the specialized nonsymmetric Macdonald polynomial $E_{w\cdot\lambda}(x_1,\ldots,x_n;q,0)$, using notation from \cite{HHL08}. The specialized nonsymmetric Macdonald polynomials can be defined combinatorially as the generating polynomial of \emph{semi-standard key tabloids}, defined below following notation in \cite{Ass18,AG21}, and so this is the natural choice for the underlying set of the Demazure crystal basis $\aB_w(\lambda)$.

The \newword{diagram} of a weak composition $\alpha$ has $\alpha_i$ cells left-justified in row $i$, indexed in coordinate notation with row $1$ on the bottom and column $1$ on the left. A \newword{filling} of a diagram is an assignment of positive integers as entries of cells of the diagram. Two cells are \newword{attacking} if they lie in the same column or lie in adjacent columns with the cell on the left strictly higher than the cell on the right. A filling is \newword{non-attacking} if no two cells with the same entry are attacking. The \newword{basement cells} lie to the left of the first column and have entry equal to their row index. All fillings in this paper are non-attacking, including the basement cells.

A \newword{triple} is a collection of three cells, possibly including basement cells, with two row adjacent and either (Type I) the third cell is above the left and the lower row is strictly longer, or (Type II) the third cell is below the right and the higher row is weakly longer. The \newword{orientation} of a triple is determined by reading the entries of the cells from smallest to largest, where a basement cell has entry equal to its row index and if two cells have equal entry we regard the one on the right as smaller. A \newword{co-inversion triple} is a Type I triple oriented counterclockwise or a Type II triple oriented clockwise, as illustrated in Fig.~\ref{fig:inv}. 

\begin{figure}[ht]
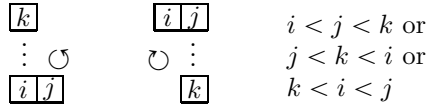

  \begin{displaymath}
      \begin{array}{l}
        \tableau{ k } \\[-0.5\cellsize] \hspace{0.4\cellsize} \vdots \hspace{0.5\cellsize} \circlearrowleft \\ \tableau{ i & j } \\
        \mathrm{Type \ I}
      \end{array}\hspace{2\cellsize}
      \begin{array}{r}
        \tableau{ i & j } \\[-0.5\cellsize] \circlearrowright \hspace{0.5\cellsize} \vdots \hspace{0.4\cellsize} \\ \tableau{ k } \\
        \mathrm{Type \ II}
      \end{array}\hspace{2\cellsize}
      \begin{array}{l}
        i < j < k \text{ or} \\
        j < k < i \text{ or} \\
        k < i < j
      \end{array}
  \end{displaymath}
  \caption{\label{fig:inv}The positions and orientation for co-inversion triples.}
\end{figure}

Given a weak composition $\alpha$, a \newword{semistandard key tabloid} of shape $\alpha$ is a non-attacking filling of the diagram of $\alpha$ with no co-inversion triples. Denote the set of semistandard key tabloids of shape $\alpha$ by $\SSKD(\alpha)$.

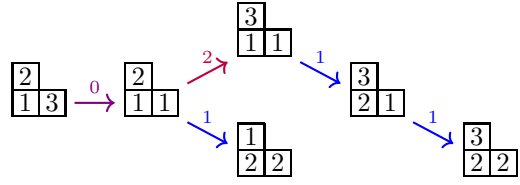
\begin{figure}[ht]
  \begin{displaymath}
    \begin{tikzpicture}[xscale=1.5,yscale=0.8]
      \node at (0,1)   (A) {$\vline\tableau{2 \\ 1 & 3 \\ & }$};
      \node at (1,1)   (B) {$\vline\tableau{2 \\ 1 & 1 \\ & }$};
      \node at (2,2)   (C) {$\vline\tableau{3 \\ 1 & 1 \\ &}$};
      \node at (2,0)   (D) {$\vline\tableau{1 \\ 2 & 2 \\ & }$};
      \node at (3,1)   (E) {$\vline\tableau{3 \\ 2 & 1 \\ & }$};
      \node at (4,0)   (F) {$\vline\tableau{3 \\ 2 & 2 \\ & }$};
      \draw[thick,color=violet,->] (A) -- (B) node[midway,above] {$\scriptstyle 0$} ;
      \draw[thick,color=blue  ,->] (B) -- (D) node[midway,above] {$\scriptstyle 1$} ;
      \draw[thick,color=purple,->] (B) -- (C) node[midway,above] {$\scriptstyle 2$} ;
      \draw[thick,color=blue  ,->] (C) -- (E) node[midway,above] {$\scriptstyle 1$} ;
      \draw[thick,color=blue  ,->] (E) -- (F) node[midway,above] {$\scriptstyle 1$} ;
    \end{tikzpicture}
  \end{displaymath}
  \caption{\label{fig:mac-crystal}The affine Demazure crystal on $\SSKD(0,2,1)$.}
\end{figure}

The \newword{weight} of a semistandard key tabloid $T$ is the weak composition $\wt(T)$ whose $i^{th}$ part is the number of entries equal to $i$. The \newword{column reading word} $w(T)$ of a semistandard key tabloid $T$ is the word obtained by reading the entries of $T$ up the columns from left to right. For instance, the reading word of the leftmost tabloid in Fig.~\ref{fig:mac-crystal} is $123$. Each tabloid is uniquely determined by its reading word.

In \cite{AG21} the authors define finite raising and lowering operators on semi-standard key tabloids of shape $\alpha$ giving rise to a finite Demazure crystal structure on $\SSKD(\alpha)$. We recall these operators here; see \cite{AG21}(Section 5) for further details. 

\begin{definition}\cite{AG21}
  For $T \in \SSKD(\alpha)$ and $1 \leq i <n$, we \newword{i-pair} the cells of $T$ with entries $i$ or $i+1$ as follows: $i$-pair $i$ and $i+1$ whenever they occur in the same column, and then iteratively $i$-pair an unpaired $i+1$ with an unpaired $i$ to its left whenever all entries $i$ or $i+1$ lying between them are already $i$-paired.
  \label{def:pair}
\end{definition}

\begin{definition}\cite{AG21} \label{def:lower-key}
  For $i \geq 1$, the \newword{lowering operator} $f_i$ acts on $T\in\SSKD(\alpha)$ by
  \begin{itemize}
  \item if all entries $i$ of $T$ are $i$-paired or if the leftmost unpaired $i$ is in row $i$ and all columns to its left have an $i$ in the same row and an $i+1$ above, then $f_i(T)=0$;
  \item otherwise, $f_i$ changes the leftmost unpaired $i$ to $i+1$ and
    \begin{itemize}
    \item swaps the entries $i$ and $i+1$ in each of the consecutive columns left of this entry that have an $i$ in the same row and an $i+1$ above, and
    \item swaps the entries $i$ and $i+1$ in each of the consecutive columns right of this entry that have an $i$ in the same row and an $i+1$ below.
    \end{itemize}
  \end{itemize}
\end{definition}

\begin{figure}[ht]
  \begin{center}
    \begin{tikzpicture}[xscale=3.75,yscale=1]
      \node at (0,0) (A) {$\vline\tableau{
          5 & \mathbf{\color{red}3} & 1 \\
          \mathbf{\color{red}3} & \mathbf{\color{red}2} & \cirfy{\mathbf{\color{blue}2}} & \mathbf{\color{red}2} & 5 & 5 & 5 & 2 \\ \\
          \mathbf{\color{red}2} & 1 & 4 & 4 & \mathbf{\color{red}3} & 1 & 1 \\ & }$};
      \node at (1,0) (B) {$\vline\tableau{
          5 & \mathbf{\color{red}2} & 1 \\
          \mathbf{\color{red}3} & \mathbf{\color{red}3} & 3 & \mathbf{\color{red}2} & 5 & 5 & 5 & \cirfy{\mathbf{\color{blue}2}} \\ \\
          \mathbf{\color{red}2} & 1 & 4 & 4 & \mathbf{\color{red}3} & 1 & 1 \\ & }$};
      \node at (2,0) (C) {$\vline\tableau{
          5 & \mathbf{\color{red}2} & 1 \\
          \mathbf{\color{red}3} & \mathbf{\color{red}3} & 3 & \mathbf{\color{red}2} & 5 & 5 & 5 & 3 \\ \\
          \mathbf{\color{red}2} & 1 & 4 & 4 & \mathbf{\color{red}3} & 1 & 1 \\ & }$};
      \node at (2.65,0) (D) {$0$};
      \draw[thick,color=purple,->] (A) -- (B) node[midway,above] {$\scriptstyle f_2$} ;
      \draw[thick,color=purple,->] (B) -- (C) node[midway,above] {$\scriptstyle f_2$} ;
      \draw[thick,color=purple,->] (C) -- (D) node[midway,above] {$\scriptstyle f_2$} ;      
    \end{tikzpicture}
  \end{center}
  \caption{\label{fig:fSSKD}The $2$-string for a semistandard key tabloid, with $2$-paired letters in red and the leftmost unpaired $2$ circled.}
\end{figure}
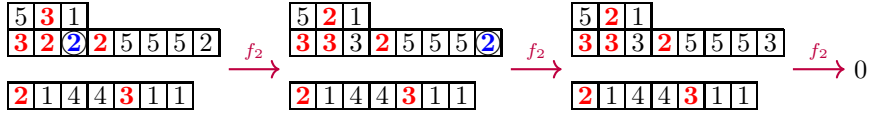

\begin{definition}\cite{AG21}  \label{def:raise-tabloid}
  For $i\geq 1$, the \newword{raising operator} $e_i$ acts on $T\in\SSKD(\alpha)$ by
  \begin{itemize}
  \item if all entries $i+1$ of $T$ are $i$-paired, then $e_i(T)=0$;
  \item otherwise, $e_i$ changes the rightmost unpaired $i+1$ to $i$ and
    \begin{itemize}
    \item swaps the entries $i$ and $i+1$ in each of the consecutive columns left of this entry that have an $i+1$ in the same row and an $i$ above, and
    \item swaps the entries $i$ and $i+1$ in each of the consecutive columns right of this entry that have an $i+1$ in the same row and an $i$ below.
    \end{itemize}
  \end{itemize}
\end{definition}

When nonzero, these operators are well-defined inverses of each other. 

\begin{theorem}\cite{AG21}
  For $1 \leq i <n$, the raising and lowering operators are well-defined maps $e_i,f_i : \SSKD(\alpha) \rightarrow \SSKD(\alpha) \cup \{0\}$ such that for $S,T\in\SSKD(\alpha)$, we have $e_i(S) = T$ if and only if $f_i(T) = S$.
  \label{thm:well-defined}
\end{theorem}

In an analogous manner, we extend these notions to the affine setting. 

\begin{definition}
  For $T \in \SSKD(\alpha)$, we \newword{affine $0$-pair} the cells of $T$ with entries $n$ or $1$ as follows: $0$-pair together $n$ and $1$ whenever they occur in the same column, and then iteratively $0$-pair an unpaired $n$ with an unpaired $1$ to its right whenever all entries $n$ or $1$ that lie between them are already $0$-paired.
  \label{def:0-pair}
\end{definition}

\begin{definition}  \label{def:affine-lowering-tabloid}
  The \newword{affine lowering operator} $f_0$ acts on $T \in \SSKD(\alpha)$ by
  \begin{itemize}
  \item if all entries $n$ of $T$ are $0$-paired or if the leftmost unpaired $n$ is in the leftmost column of $T$ and $\alpha$ has more than one nonzero part, then $f_0(T)=0$;
  \item otherwise, $f_0$ changes the leftmost unpaired $n$ to $1$ and
    \begin{itemize}
    \item swaps the entries $n$ and $1$ in each of the consecutive columns left of this entry that have an $n$ in the same row and a $1$ above, and
    \item swaps the entries $n$ and $1$ in each of the consecutive columns right of this entry that have an $n$ in the same row and a $1$ below.
    \end{itemize}
  \end{itemize}
\end{definition}

\begin{figure}[ht]
  \begin{center}
    \begin{tikzpicture}[xscale=3.75,yscale=1]
      \node at (0,0) (A) {$\vline\tableau{
          \mathbf{\color{red}5} & 2 & \cirfy{\mathbf{\color{blue}5}} \\
          3 & 3 & 3 & 2 & 5 & \mathbf{\color{red}5} & \mathbf{\color{red}5} & 2 \\ \\
          2 & \mathbf{\color{red}1} & 4 & 4 & 3 & \mathbf{\color{red}1} & \mathbf{\color{red}1} \\ & }$};
      \node at (1,0) (B) {$\vline\tableau{
          \mathbf{\color{red}5} & 2 & 1 \\
          3 & 3 & 3 & 2 & \cirfy{\mathbf{\color{blue}5}} & \mathbf{\color{red}5} & \mathbf{\color{red}5} & 2 \\ \\
          2 & \mathbf{\color{red}1} & 4 & 4 & 3 & \mathbf{\color{red}1} & \mathbf{\color{red}1} \\ & }$};
      \node at (2,0) (C) {$\vline\tableau{
          \mathbf{\color{red}5} & 2 & 1 \\ 
          3 & 3 & 3 & 2 & 1 & \mathbf{\color{red}1} & \mathbf{\color{red}1} & 2 \\ \\
          2 & \mathbf{\color{red}1} & 4 & 4 & 3 & \mathbf{\color{red}5} & \mathbf{\color{red}5} \\ & }$};
      \node at (2.65,0) (D) {$0$};
      \draw[thick,color=violet,->] (A) -- (B) node[midway,above] {$\scriptstyle f_0$} ;
      \draw[thick,color=violet,->] (B) -- (C) node[midway,above] {$\scriptstyle f_0$} ;
      \draw[thick,color=violet,->] (C) -- (D) node[midway,above] {$\scriptstyle f_0$} ;      
    \end{tikzpicture}
  \end{center}
  \caption{\label{fig:aSSKD}The $0$-string for a semistandard key tabloid, with $0$-paired letters in red and the leftmost unpaired $5$ circled.}
\end{figure}
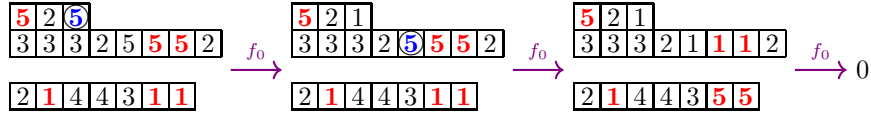

Notice if $f_0$ swaps entries $n$ and $1$ in $T$, then those two are in the same column and so are $0$-paired to one another both in $T$ and in $f_0(T)$. Thus since $f_0$ acts on the leftmost unpaired $n$, it neither creates nor destroys a $0$-pairing. Moreover, analogous to Lemmas 5.6 and 5.7 in \cite{AG21}, when $f_0$ swaps $0$-paired entries, we obtain information about their relative row lengths as follows.

\begin{lemma}\label{lem:pairing-rows}
  Let $T \in \SSKD(\alpha)$ such that $f_0(T) \neq 0$. Suppose cells $x,y$ are $0$-paired together in $T$, say with $x$ having entry $n$ in $T$ and entry $1$ in $f_0(T)$. Then
  \begin{itemize}
  \item if $y$ lies above $x$, then the row of $x$ is strictly longer than the row of $y$; 
  \item if $y$ lies below $x$, then the row of $x$ is weakly longer than the row of $y$. 
  \end{itemize}
\end{lemma}

\begin{proof}
The proofs are entirely analogous to those in Lemma 5.6 and 5.7 in \cite{AG21}. 
\end{proof}

Using this, we show the affine lowering operator is well-defined on $\SSKD(\alpha)$.

\begin{theorem}
  For $T\in\SSKD(\alpha)$, if $f_0(T)\neq 0$, then $f_0(T)\in\SSKD(\alpha)$.
  \label{thm:aff-well-defined}
\end{theorem}

\begin{proof}
  For $T\in\SSKD(\alpha)$, if $f_0(T)\neq 0$, then $f_0(T)$ is a filling of shape $\alpha$. To show $f_0(T)\in\SSKD(\alpha)$, we must show no cells of $f_0(T)$ are attacking and $f_0(T)$ has no co-inversion triples. We take each in turn.

  From its definition, $f_0$ will never create attacking cells within the same column. Suppose we have two cells, $y$ above and left of $x$, in attacking position with entries $1$ or entries $n$ and $f_0$ acts nontrivially on at least one the cells, since otherwise they remain non-attacking. If, in $T$, $y$ has entry $n$ and $x$ has entry $1$, then, since these cells lie in the same column, $f_0$ acts nontrivially only if it changes both entries, in which case they still have different entries in $f_0(T)$. Otherwise, in $T$, $y$ has entry $1$ and $x$ has entry $n$. If the row of $y$ is weakly longer, then the cell above $x$ and immediately right of $y$ must have entry $1$ and thus, $x$ will be $0$-paired with a $1$ above it. By Lemma~\ref{lem:pairing-rows}, this implies $f_0$ does not change the entry in $x$. However, by the way $f_0$ is defined, this in turn implies $f_0$ also does not change the entry of $y$, and so no attacking pair is created. If instead, the row of $y$ is strictly shorter than the row of $x$, then the cell immediately to the left of $x$ must also have entry $n$ and thus, under $f_0$ the values of these three cells get exchanged, avoiding the creation of attacking cells. Thus $f_0(T)$ is non-attacking.
  
  Now consider three cells forming a triple, labeled as in Fig.~\ref{fig:inv}. If $f_0$ changes the entries in none, one, or all of the three cells, then the orientations of the triple will be unaltered, and so no co-inversion triple is created. Thus we may assume exactly two cells of a triple are altered by the affine lowering operator. If the two affected cells have the same entry, either $1$ or $n$, then these cells lie in the same row. If both cells have entry $n$ (resp. $1$) then $f_0$ will act either on the left (resp. right) cell or on both. In either case, the correct orientations are always preserved. Thus we may assume the affected cells have different entries.

  Suppose first the triple is of Type I. If $(k,j)=(n,1)$, then by Lemma~\ref{lem:pairing-rows} $f_0$ cannot simultaneously act on both cells. However, if $f_0$ acts only on $j$, then $j$ is necessarily $0$-paired with an $n$ attacking $k$, a contradiction, and if $f_0$ acts only on $k$, then $j$ must be $0$-paired with an $n$ which is either in the same row as $k$, contradicting the fact that $k$ was in a row strictly shorter than $j$, or in a row strictly shorter than $j$, contradicting the fact that $f_0$ acted nontrivially on $k$. If $(i,j)=(1,n)$, then by Lemma~\ref{lem:pairing-rows} $f_0$ acts only on $j$ by sending it to $1$, which preserves the correct orientation. If $(k,i)=(1,n)$, then again by Lemma~\ref{lem:pairing-rows} $f_0$  acts nontrivially on the triple if and only if $j=n$, in which case all values are exchanged and no co-inversion triples are created. Lastly, if $(k,j)=(1,n)$ (resp. $(k,i)=(n,1)$) then necessarily $i=n$ (resp. $j=1$) in which case $f_0$ acts by swapping all entries, hence maintaining the correct orientations. 

  Suppose next the triple is of Type II. If $(i,j)=(1,n)$ then by Lemma~\ref{lem:pairing-rows} $f_0$ cannot act on both cells. If it acts only on $j$, then the orientation is maintained. If it acts only on $i$, then by Lemma~\ref{lem:pairing-rows}, $i$ is $0$-paired with an $n$ above it, which is impossible since this would attack $j$. If $(j,k) = (1,n)$, $(j,k) = (n,1)$, or $(i,k) = (1,n))$, then $i=1$, $i=n$, or $j=1$, respectively, in which case $f_0$ acts on all three cells by swapping their values. Hence, in all cases $f_0$ preserves the correct orientation and thus never creates co-inversion triples of any kind. Thus $f_0(T)$ has no co-inversion triples.
\end{proof}

In an entirely analogous manner, we define affine raising operators to be inverse to the affine lowering operators whenever both act nontrivially.

\begin{definition} \label{def:affine-raise-tabloid}
  The \newword{affine raising operator} $e_0$ acts on $T\in\SSKD(\alpha)$ by
  \begin{itemize}
  \item if all entries $1$ of $T$ are $0$-paired or if the rightmost unpaired $1$ is in a row with index less than $n$ and all columns to its left have a $1$ in the same row with an $n$ above, then $e_0(T)=0$;
  \item otherwise, $e_0$ changes the rightmost unpaired $1$ to $n$ and
    \begin{itemize}
    \item swaps the entries $1$ and $n$ in each of the consecutive columns left of this entry that have a $1$ in the same row and an $n$ above, and
    \item swaps the entries $1$ and $n$ in each of the consecutive columns right of this entry that have a $1$ in the same row and an $n$ below.
    \end{itemize}
  \end{itemize}
\end{definition}

\begin{proposition}
  For $S,T\in\SSKD(\alpha)$, $f_0(S) = T$ if and only if $e_0(T) = S$.
\end{proposition}

\begin{proof}
Suppose $S,T \in \SSKD(\alpha)$ and $f_0(S)=T$. Let $c$ denote the column index of the leftmost unpaired cell with value $n$ of $S$.  Since $f_0$ acts on the leftmost unpaired $n$, then all columns left of $c$ have no unpaired $n$'s and all columns right of $c$ have no unpaired $1$'s.  Hence, the action of $f_0$ on $S$ will not create new $0$-pair in $T$ since the new $1$ that is created by $f_0$ in column $c$ has no $n$ to pair with to its left. Thus, the rightmost unpaired $1$ of $S$ is precisely the same $1$ located in column $c$. Since $e_0$ acts on consecutive columns adjacent to this $1$ with $n$'s and $1$'s distributed in an inverted way to the columns on which $f_0$ acts, then $e_0(S) \neq 0$ and evidently $e_0(S)=e_0(f_0(T))=T$. The case $e_0(T)=S$ is completely analogous. 
\end{proof}

In \cite{AG21}, the authors show each connected component of the \emph{finite} crystal on $\SSKD(\alpha)$ is isomorphic to a \emph{finite} Demazure crystal $\B_w(\lambda)$ for some $\lambda\in P^{+}$ and some $w\in W$. Using the affine raising and lowering operators on semistandard key tabloids, we may define the \newword{affine tabloid crystal} for a weak composition $\alpha$ to be the set $\SSKD(\alpha)$, the weight map $\wt$, and the finite and affine raising and lowering operators. Our main result is to show this crystal is isomorphic to the affine Demazure crystal $\aB_w(\lambda)$, where $w \cdot \lambda = \alpha$.

%
\section{Crystal filtration}
%
\label{sec:embed}

To establish the isomorphism between the affine tabloid crystal and the corresponding affine Demazure crystal, we construct injective maps $\SSKD(\alpha) \hookrightarrow \SSKD(\beta)$ whenever $\alpha \preceq \beta$ in Bruhat order, meaning if we write $\alpha = u \cdot \lambda$ and $\beta = v \cdot \lambda$ with $\lambda$ minimal and $u,v$ minimal length, then $u \leq v$ in Bruhat order. We begin with the finite case.

\begin{definition}\label{def:raising-crystal-finite}
  Given a weak composition $\alpha$ and an index $1 \leq i<n$ such that $\alpha_i>\alpha_{i+1}$, the \newword{embedding map} $\D_i$ sends $T \in \SSKD(\alpha)$ to the filling $\D_i(T)$ of shape $s_i \cdot \alpha$ constructed as follows. Let $y_k, x_k$ denote the entries in rows $i+1,i$ and column $k$, respectively, for $k=0,\ldots,\alpha_{i+1}$, with $k=0$ corresponding to the basement. For $k\geq 0$, place entries $y_{k+1},x_{k+1}$ into rows $i,i+1$, column $k+1$ of $\D_i(T)$ by
  \begin{itemize}
  \item if $y_k$ is above $x_k$ in $\D_i(T)$, then place $y_{k+1}$ above $x_{k+1}$ unless doing so creates a Type II co-inversion triple for $x_{k+1},y_k,y_{k+1}$, in which case place $x_{k+1}$ above $y_{k+1}$;
  \item else $x_k$ is above $y_k$, so place $x_{k+1}$ above $y_{k+1}$ unless doing so creates a Type II co-inversion triple for $y_{k+1},x_k,x_{k+1}$, in which case place $y_{k+1}$ above $x_{k+1}$.
  \end{itemize}
  For all remaining cells, set entries of $\D_i(T)$ to match those of $T$.
\end{definition}

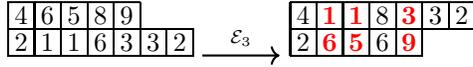
\begin{figure}[ht]
  \begin{center}
    \begin{tikzpicture}[xscale=3.75,yscale=1]
      \node at (0,0) (A) {$\vline\tableau{4 & 6 & 5 & 8 & 9 \\ 2 & 1 & 1 & 6 & 3 & 3 & 2 \\ & \\ & }$};
      \node at (1,0) (B) {$\vline\tableau{4 & \mathbf{\color{red}1} & \mathbf{\color{red}1} & 8 & \mathbf{\color{red}3} & 3 & 2 \\ 2 & \mathbf{\color{red}6} & \mathbf{\color{red}5} & 6 & \mathbf{\color{red}9} \\ & \\ & }$};
      \draw[thick,->] (A) -- (B) node[midway,above] {$\scriptstyle \D_3$} ;
    \end{tikzpicture}
  \end{center}       
  \caption{\label{fig:fswap}An example of the embedding $\D_3:\SSKD(0,0,7,5)\hookrightarrow\SSKD(0,0,5,7)$.}
\end{figure}

The embedding map is constructed precisely so the following local result holds.

\begin{lemma}\label{lem:D-local}
  Let $\alpha$ be a weak composition with $\alpha_j=0$ for $j \neq i,i+1$, and suppose $\alpha_i>\alpha_{i+1}$. For any $T \in \SSKD(\alpha)$, we have $\D_i(T) \in \SSKD(s_i \cdot \alpha)$. 
\end{lemma}

\begin{proof}
  Observe all triples are of Type I for $T$ and of Type II for $\D_i(T)$. Since $T$ is non-attacking, entries in column $k+1$ are distinct, and so changing their order reverses the orientation of the Type II triple for those cells and the upper cell in column $k$. Thus by construction $\D_i(T)$ has no co-inversion triples. Since $\D_i$ preserves the column sets, the only potential attacking cells are in consecutive columns, but two such entries, say in columns $k,k+1$, necessarily create a Type II co-inversion triple. Thus $\D_i(T)$ is non-attacking as well, and so $\D_i(T) \in \SSKD(s_i \cdot \alpha)$.  
\end{proof}

Before establishing more properties of $\D_i$, we present the affine case.

\begin{definition}\label{def:raising-crystal-affine}
  Given $\alpha$ such that $\alpha_n\geq \alpha_{1}>0$, the \newword{affine embedding map} $\D_0$ sends $T \in \SSKD(\alpha)$ to the filling $\D_0(T)$ of shape $s_0 \cdot \alpha$ constructed as follows. Let $y_k, x_k$ denote the entries in rows $n,1$ and column $k$, respectively, for $k=0,\ldots,\alpha_{1}$, with $k=0$ corresponding to the basement. For $k\geq 1$ place entries $y_{k},x_{k+1}$ into the cells in row $n$, column $k$ and row $1$, column $k+1$ of $\D_0(T)$ by
  \begin{itemize}
  \item  if $x_{k}$ is in row $1$, column $k$ of $\D_0(T)$, then place $y_{k}$ above it in row $n$ unless doing so creates a Type I co-inversion triple or attacking pair for $x_k, x_{k+1}, y_{k}$, in which case place $x_{k+1}$ in row $n$ above $x_{k}$;
  
  \item else $y_{k-1}$ is in row $1$, column $k$, so place $x_{k+1}$ above it in row $n$ unless doing so creates a Type I co-inversion triple or attacking pair for $y_{k-1},y_{k},x_{k+1}$, in which case place $y_{k}$ in row $n$ above $y_{k-1}$.
  
  \end{itemize}
  For all remaining cells, set entries of $\D_0(T)$ to match those of $T$, then remove the entries in row $n$, columns $\alpha_1,\ldots,\alpha_{n}$ and append them to the end of row $1$.
\end{definition}

\begin{figure}[ht]
  \begin{center}
    \begin{tikzpicture}[xscale=3.5,yscale=1]
      \node at (0,0) (A) {$\vline\tableau{3 & 3 & 8 & 6 & 3 & 3 \\ & \\ & \\ 1 & 4 & 5 & 5 & 7 & 2}$};
      \node at (1,0) (B) {$\vline\tableau{3 & \mathbf{\color{red}5} & \mathbf{\color{red}5} & 6 & \mathbf{\color{red}2} \\ & \\ & \\ 1 & 4 & \mathbf{\color{red}3} & \mathbf{\color{red}8} & 7 & \mathbf{\color{red}3} & 3}$};
      \draw[thick,->] (A) -- (B) node[midway,above] {$\scriptstyle \D_0$} ;
      \draw[thin] (0.71,0.35) -- (0.82,-0.35) ;
      \draw[thin,color=red] (0.81,0.35) -- (0.92,-0.35) ;
      \draw[thin,color=red] (0.91,0.35) -- (1.02,-0.35) ;
      \draw[thin] (1.01,0.35) -- (1.12,-0.35) ;
      \draw[thin,color=red] (1.11,0.35) -- (1.22,-0.35) ;
    \end{tikzpicture}
  \end{center}       
  \caption{\label{fig:aswap}An example of the embedding $\D_0:\SSKD(6,0,0,6)\hookrightarrow\SSKD(7,0,0,5)$.}
\end{figure}
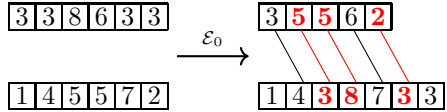

As in the finite case, we have the following affine analog of Lemma~\ref{lem:D-local}.

\begin{lemma}\label{lem:D-local-affine}
  Let $\alpha$ be a weak composition with $\alpha_j=0$ for $j \neq 1,n$, and suppose $\alpha_n \geq \alpha_{1}>0$. For any $T \in \SSKD(\alpha)$, we have $\D_0(T) \in \SSKD(s_0 \cdot \alpha)$. 
\end{lemma}

\begin{proof}
  Observe all triples are of Type II for $T$ and of Type I for $\D_0(T)$. Since $T$ is non-attacking, entries along a northwest to southeast diagonal are distinct, and so changing their order reverses the orientation of the Type I triple for those cells and their southwest neighbor. Thus by construction $\D_0(T)$ has no co-inversion triples. Since $\D_0$ preserves the diagonal sets, the only potential attacking cells are in the same column, but two such entries necessarily create a Type I co-inversion triple. Thus $\D_0(T)$ is non-attacking as well, and so $\D_0(T) \in \SSKD(s_0 \cdot \alpha)$.  
\end{proof}

The embedding maps are defined so that locally there are no co-inversion triples. This property holds globally as well, showing we have a map of tabloids.

\begin{theorem}\label{thm:embedding-well-defined}
   If $\alpha \preceq s_i \cdot \alpha$, then $\D_i(T)\in \SSKD(s_i\cdot \alpha)$ for any $T\in\SSKD(\alpha)$.
\end{theorem}

\begin{proof}
  We begin with the non-attacking condition. For $i>0$, $\D_i$ acts only on rows $i$ and $i+1$ of $T$ and, by Lemma~\ref{lem:D-local}, does not create attacking cells within these two rows. Thus $\D_i$ does not create attacking cells between rows $i,i+1$ and any row above or below them. Similarly for $i=0$, $\D_0$ acts only on rows $1$ and $n$ of $T$ and, by Lemma~\ref{lem:D-local-affine}, does not create attacking cells within these two rows. Two cells are attacking exactly when they are within $n$ letters in the column reading word of $T$ (accounting for empty rows), and so $\D_0$ does not create attacking cells between rows $1,n$ and any row in between them. Thus $\D_i(T)$ is non-attacking for $i\geq 0$.

  By the local nature of the embedding maps, $\D_i(T)$ has no co-inversion triples between two unaffected rows. By Lemmas~\ref{lem:D-local} and \ref{lem:D-local-affine}, $\D_i(T)$ has no co-inversion triples between the two affected rows. Suppose, for contradiction, $\D_i(T)$ has a co-inversion triple between one of the affected rows, say $j$, and an unaffected row, say $k$. By \cite[Proposition~2.9]{AG21}, the co-inversion triple must involve the basement. Any Type I co-inversion triple involving the basement implies an attacking pair as well, so we may assume any co-inversion triple is Type II.

  Suppose first $i>0$ and $k>i+1$, so that row $k$ is weakly longer than row $j$ in $\D_i(T)$. Denote the entries in column $1$ rows $k,i+1,i$ of $\D_i(T)$ by $z,y,x$, respectively. We must have $\alpha_i > \alpha_k \geq \alpha_{i+1}$, since otherwise rows $i,i+1$ are both weakly shorter than row $k$, and so $T$ must also have a Type II co-inversion triple. Thus $j=i$, and $x$ must have been in row $i$ of $T$ as well so as to avoid creating the same triple in $T$. Now, since $\D_i$ did not swap $x$ and $y$, we must have $x < y \leq i+1 < k$. Thus the alleged co-inversion triple $x,z,k$ in $\D_i(T)$ must have $z < x < k$, which then forces the triple $y,z,k$ to be a Type II co-inversion triple in $T$, a contradiction.

  Next suppose $i>0$ and $k<i$, so that row $k$ is weakly shorter than row $j$ in $\D_i(T)$. Denote the entries in column $1$ rows $k,i,i+1$ of $\D_i(T)$ by $x,y,z$, respectively. Similar to before, we must have $\alpha_i \geq \alpha_k > \alpha_{i+1}$, since otherwise rows $i,i+1$ are both weakly longer than row $k$, and so $T$ must also have a Type II co-inversion triple. Thus $j=i+1$, and $z$ must have been in row $i+1$ of $T$ as well so as to avoid creating the same triple in $T$. Since $\D_i$ did not swap $y$ and $z$, we must have $y < z \leq i+1$. Thus the alleged co-inversion triple $x,z,i+1$ in $\D_i(T)$ must have $z < x < i+1$, which then forces the triple $x,y,i+1$ to be a Type II co-inversion triple in $T$, a contradiction.

  Finally suppose $i=0$. If $j=1$, then the entry in column $1$, row $1$ is unchanged and row $1$ is longer in $\D_0(T)$ than in $T$, so the same cells form a Type II co-inversion triple in $T$, a contradiction. Thus $j=n$, and since row $n$ is shorter in $\D_0(T)$ than in $T$, we must the entry, say $z$ in column $1$ of row $n$ of $\D_0(T)$ was previously in column $2$, row $1$ of $T$. Let $x$ denote the entry in column $1$, row $k$. Since row $n$ of $\D_0(T)$ is weakly longer than row $k$, we must have $\alpha_1 > \alpha_k$. Thus to avoid a Type I triple for $1,z,x$ in $T$, we must have $1 < x < z$. However, to have a Type II triple for $x,z,n$ in $\D_0(T)$, we must have $z < x < n$, a contradiction. 
\end{proof}

We next show the embedding map $\D_i$ is injective and the heads of all $i$-strings are contained in the image. Thus the only elements not in the image of $\D_i$ are those obtained by following $i$-strings from elements in the image.

\begin{lemma}\label{lem:injective}
  For $\alpha \preceq s_i \cdot \alpha$, the embedding map $\D_i$ is injective. Moreover, for $T\in\SSKD(s_i \cdot \alpha)$, if $e_i(T)=0$ then there exists $S\in\SSKD(\alpha)$ such that $\D_i(S)=T$.
\end{lemma}

\begin{proof}
  The embedding map is reversed by considering the opposite type of triples and working right to left. In particular, if $\D_i(S)=\D_i(S')$ for two tabloids $S,S' \in \SSKD(\alpha)$, then by reversing the embedding map as described above, the preimage of $\D_i(S)$ and $\D_i(S')$ will have a unique column set. Hence, $S$ and $S'$ have the same column set and thus by \cite[Proposition~2.9]{AG21} must be the same tabloid. 
  
  Now suppose that $T \in \SSKD(s_i\cdot \alpha)$ satisfying $e_i(T)=0$.  For $i \neq 0$ this implies that every $i+1$ appearing in $T$ has an $i$ below it or in a column to its left. Let $S$ denote the preimage of $T$ upon reversing the embedding map $\D_i$. The only way in which a co-inversion triple of Type I with the basement in $S$ could arise, and thus prevent $S$ from being in $\SSKD(\alpha)$,  is if in $T$ the cell in row $i+1$ of the first column has value $i+1$. Since $e_i(T)=0$ this implies that the cell immediately below it has value $i$. This in turn causes the cell in column $2$ in $S$ to the right of $i+1$ to have value $i+1$. Hence, in $T$ the second column has an $i$ in row $i$ and an $i+1$ in row $i+1$. Using the fact that $e_i(T) =0$ and iterating this procedure by passing back and forth from $S$ to $T$ it can be deduced that in $T$ all the entries in row $i+i$ have value $i+1$ and in row $i$ have value $i$. Since $\alpha \preceq s_i \cdot \alpha$ this means $T$ has an unpaired $i+1$ which contradicts the fact that $e_i(T)=0$. Hence $i$ and $i+1$ do not get flipped in $S$, thus $S$ is indeed in $\SSKD(\alpha)$. 
  
  If $i=0$ then since the first entry in row $1$ must have value $1$ and $\D_0$ does not modify this cell at all, then after reversing the affine embedding map no matter what value is placed in the first entry of row $n$, this triple will always have the correct orientation with respect to the basement, thus $S$ will always be in $\SSKD(\alpha)$. 
\end{proof}

We now show the embedding maps respect the crystal structure.

\begin{theorem}\label{thm:intertwine}
  For $\alpha \preceq s_i \cdot \alpha$ and $T \in \SSKD(\alpha)$, we have $\varphi_j(T)=\varphi_j(\D_i(T))$ for any $0\leq j <n$ and if $f_j(T) \neq 0$ for some $j\geq 0$, then $\D_i(f_j(T)) = f_j(\D_i(T))$.
\end{theorem}

\begin{proof}
  Let $T \in \SSKD(\alpha)$. By \cite[(3.16)]{AG21}, $\varphi_j(T)$ for a tabloid $T$ counts the number of cells of $T$ with entry $j$ that are not $j$-paired (or $0$-paired in the case $j=n$). For $i \neq 0$, the column sets of $T$ and $\D_i(T)$ coincide, hence the $j$-pairing of both tableaux will be the same for any $j$. For $i=0$, although the location of $j$-paired cells might be modified, the net quantity of pairs remains intact under $\D_0$. Thus, $\varphi_j(T)=\varphi_j(\D_i(T))$ for any $0\leq j <n$ as claimed.

  Now suppose $f_j(T) \neq 0$. Then $\varphi_j(T)>0$ and so $\varphi_j(\D_i(T))>0$ as well. In the finite embedding case, where $i>0$, $T$ and $\D_i(T)$ have the same column sets, and so $f_j$ will act on the same column for both. This ensures $f_j(T)$ and $f_j(\D_i(T))$ have the same column set. Since $\D_i(f_j(T))$ and $f_j(T)$ have the same column sets by definition of $\D_i$, we conclude $f_j(\D_i(T))$ and $\D_i(f_j(T))$ also agree on column sets, and so $f_j(\D_i(T)) = \D_i(f_j(T))$ by \cite[Proposition~2.9]{AG21}. Similarly for the affine embedding $\D_0$, the preservation of the $j$-pairing rule and the fact that $j,j+1$ compare the same with all letters $k\neq j,j+1$ ensures $f_j(\D_0(T))$ and $\D_0(f_j(T))$ have the same column sets. Thus \cite[Proposition~2.9]{AG21} again ensures $f_j(\D_i(T)) = \D_i(f_j(T))$. 
\end{proof}

\begin{figure}[ht]
\begin{tikzpicture}[scale=1,every node/.style={scale=1}]

\begin{scope}[shift={(-4.5,0)}]
         \node at (0,0) (A) {$\vline\tableau{3\\ 2\\1&1&2}$};
         \node at (0,-2) (B) {$\vline\tableau{3\\ 2\\1&1&3}$};
         \node at (1,-4) (D) {$\vline\tableau{3\\ 2\\1&1&1}$};
      \draw[thick, red, ->] (A) -- (B) node[midway,right] {$\color{red}f_2$} ;
      \draw[thick, purple, ->] (B) -- (D) node[midway,right] {$\color{purple}f_0$} ;
  \end{scope}
\draw [thick, right hook->] (-3,-2) -- (-1.5,-2) node[midway, above]{$\D_1$};  

\begin{scope}[xscale=1]
         \node at (0,0) (A) {$\vline\tableau{3\\ 2&1&2\\1}$};
         \node at (0,-2) (B) {$\vline\tableau{3\\ 2&1&3\\1}$};
         \node at (-1,-4) (C) {$\vline\tableau{3\\ 2&2&3\\1}$};
         \node at (1,-4) (D) {$\vline\tableau{3\\ 2&1&1\\1}$};
         \node at (1,-6) (F) {$\vline\tableau{3\\ 2&2&1\\1}$};
         \node at (1,-8) (H) {$\vline\tableau{3\\ 2&2&2\\1}$};
      \draw[thick, red, ->] (A) -- (B) node[midway,right] {$\color{red}f_2$} ;
      \draw[thick, blue, ->] (B) -- (C) node[midway,right] {$\color{blue}f_1$} ;
      \draw[thick, purple, ->] (B) -- (D) node[midway,right] {$\color{purple}f_0$} ;
      \draw[thick, blue, ->] (D) -- (F) node[midway,right] {$\color{blue}f_1$} ;
      \draw[thick, blue, ->] (F) -- (H) node[midway,right] {$\color{blue}f_1$} ;
  \end{scope}
\draw [thick, right hook->] (1.5,-2) -- (3,-2) node[midway, above]{$\D_0$};  
  
\begin{scope}[shift={(5,0)}, xscale=1, yscale=1]
         \node at (0,0) (A) {$\vline\tableau{&\\ 2&1&2\\1&3}$};
         \node at (0,-2) (B) {$\vline\tableau{&\\ 2&1&3\\1&3}$};
         \node at (-1,-4) (C) {$\vline\tableau{&\\ 2&2&3\\1&3}$};
         \node at (1,-4) (D) {$\vline\tableau{&\\ 2&1&1\\1&3}$};
         \node at (-1,-6) (E) {$\vline\tableau{&\\ 2&2&3\\1&1}$};
         \node at (1,-6) (F) {$\vline\tableau{&\\ 2&2&1\\1&3}$};
         \node at (-1,-8) (G) {$\vline\tableau{&\\ 2&2&1\\1&1}$};
         \node at (1,-8) (H) {$\vline\tableau{&\\ 2&2&2\\1&3}$};
         \node at (0,-10) (I) {$\vline\tableau{&\\ 2&2&2\\1&1}$};
      \draw[thick, red, ->] (A) -- (B) node[midway,right] {$\color{red}f_2$} ;
      \draw[thick, blue, ->] (B) -- (C) node[midway,right] {$\color{blue}f_1$} ;
      \draw[thick, purple, ->] (B) -- (D) node[midway,right] {$\color{purple}f_0$} ;
      \draw[thick, purple, ->] (C) -- (E) node[midway,right] {$\color{ purple}f_0$} ;
      \draw[thick, blue, ->] (D) -- (F) node[midway,right] {$\color{blue}f_1$} ;
      \draw[thick,  purple, ->] (E) -- (G) node[midway,right] {$\color{ purple}f_0$} ;
      \draw[thick, blue, ->] (F) -- (H) node[midway,right] {$\color{blue}f_1$} ;
      \draw[thick, blue, ->] (G) -- (I) node[midway,right] {$\color{blue}f_1$} ;
      \draw[thick,  purple, ->] (H) -- (I) node[midway,right] {$\color{ purple}f_0$} ;
      \end{scope}
  \end{tikzpicture}
  \caption{Examples of embedding maps for $\eta_{3,5} = (2,2,1)$, giving the filtration $\tilde{\mathcal{B}}_{s_0s_2}(\eta_{3,5}) \subset \tilde{\mathcal{B}}_{s_1s_0s_2}(\eta_{3,5}) \subset \tilde{\mathcal{B}}_{s_0s_1s_0s_2}(\eta_{3,5})$.}\label{fig:complete example}
\end{figure}
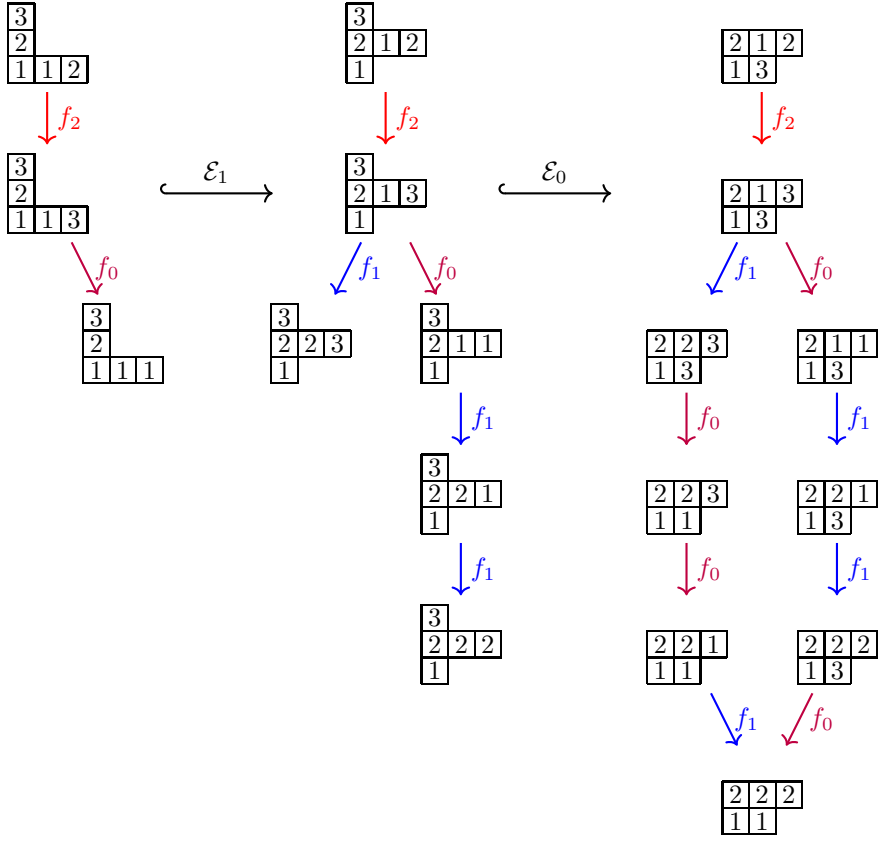

For example, Fig.~\ref{fig:complete example} shows the embedding for $\eta_{3,5} =(2,2,1)$ from $\tilde{\mathcal{B}}_{s_0s_2}(\eta_{3,5})$ to $\tilde{\mathcal{B}}_{s_1s_0s_2}(\eta_{3,5})$ via $\D_1$ and then into $\tilde{\mathcal{B}}_{s_0s_1s_0s_2}(\eta_{3,5})$ via $\D_0$. These embeddings display the Bruhat filtrations of the Demazure modules on their corresponding crystals. At each step $\D_i$ enlarges the preceeding crystal by complete $i$-strings only. 

We now present our main theorem, an explicit realization of the affine Demazure crystal on semistandard key tabloids.

\begin{theorem}\label{thm:isomorphism}
  For any weak composition $\alpha$ with $\alpha = w\cdot \eta_{n,k}$ and $w$ minimal length, there is a weight-preserving bijection $\theta_{\alpha} : \SSKD(\alpha) \rightarrow \aB_w(\eta_{n,k})$ that intertwines the crystal operators. That is, for $T\in\SSKD(\alpha)$, we have $f_i(T) \neq 0$ if and only if $f_i(\theta_{\alpha}(T)) \neq 0$ and, in this case, $f_i(\theta_{\alpha}(T))=\theta_{\alpha}(f_i(T))$.
\end{theorem}

\begin{proof}
  For $w$ length $0$, observe $\SSKD(\eta_{n,k})$ contains the single tabloid $U_{\eta_{n,k}}$ of shape $\eta_{n,k}$ with all entries equal to their row index. The column reading word (bottom to top, left to right) of $U_{\eta_{n,k}}$ coincides with the word for affine highest weight element $\tilde{u}_{\eta_{n,k}}$. Thus $U_{\eta_{n,k}}$ and $\tilde{u}_{\eta_{n,k}}$ have the same weight and, moreover, $\varphi_i(\tilde{u}_{\eta_{n,k}}) = \varphi_i(U_{\eta_{n,k}})$ for all $i$. Thus we proceed by induction on the length of $w$, assuming a weight-preserving bijection $\theta_{\alpha} : \SSKD(\alpha) \rightarrow \aB_w(\eta_{n,k})$ that intertwines the crystal operators and that preserves string lengths, and we consider $s_i \cdot \alpha \succ \alpha$.

  The Demazure operator $\mathfrak{D}_i$ gives the inclusion $\aB_w(\eta_{n,k}) \subset \aB_{s_i w}(\eta_{n,k})$ where every $b' \in \aB_{s_i w}(\eta_{n,k}) \setminus \aB_w(\eta_{n,k})$ can be written uniquely as $b' = f_i^k(b)$ for some $b\in\aB_w(\eta_{n,k})$ and $k>0$. By Theorem~\ref{thm:embedding-well-defined}, the embedding map $\D_i$ gives the inclusion $\D_i(\SSKD(\alpha)) \subset \SSKD(s_i \cdot \alpha)$. By Lemma~\ref{lem:injective}, every $T' \in \SSKD(s_i \cdot \alpha) \setminus \D_i(\SSKD(\alpha))$ can be written uniquely as $T' = f_i^k(\D_i(T))$ for some $T\in\SSKD(\alpha)$ and $k>0$. By induction, if $\theta_{\alpha}(T) = b$, then $\varphi_i(T) = \varphi_i(b)$, and so by Theorem~\ref{thm:intertwine}, $\varphi_i(\D_i(T)) = \varphi_i(b)$ as well. Thus we have a weight-preserving bijection between $\SSKD(s_i \cdot \alpha) \setminus \D_i(\SSKD(\alpha))$ and $\aB_{s_i w}(\eta_{n,k}) \setminus \aB_w(\eta_{n,k})$. By Theorem~\ref{thm:intertwine}, the crystal operators intertwine with $\D_i$, and so we may extend $\theta_{\alpha}$ to a bijection $\theta_{s_i \cdot \alpha}$ as claimed. 
\end{proof}

Notice in the proof above we make explicit use of the affine highest weight element $\tilde{u}_{\eta_{n,k}}$. In fact, we can describe the image of $\tilde{u}_{\eta_{n,k}}$ in $\SSKD(\alpha)$ for any suitable weak composition $\alpha$ as follows.

\begin{proposition}
  For a weak composition $\alpha$ of length $n$, define $\tilde{U}_{\alpha}$ to be the filling of the diagram of $\alpha$ by columns left to right, bottom to top with entries $1, \ldots, n$ repeating as needed. Then $\tilde{U}_{\alpha} \in \SSKD(\alpha)$ with $\wt(\tilde{U}_{\alpha}) = \eta_{n,k}$ where $k = |\alpha|$.
  \label{prop:hwt}
\end{proposition}

\begin{proof}
By definition the reading word of $\tilde{U}_{\alpha}$ is $(12\dots n)^m 12\dots r$ where $k = mn+r$ is the unique decomposition described beneath equation \eqref{eq:eta}. It clearly follows that wt$(\tilde{U}_{\alpha})$ is exactly $\eta_{n,k}$. 

Now, since each column of $\tilde{U}_{\alpha}$ has at most $n$ rows then by construction no column will have the same entry twice. Moreover, if a value $x$ occurs in both columns $c$ and $c+1$ of $\tilde{U}_{\alpha}$, since the number of boxes weakly above $x$ in column $c$ plus the number of boxes strictly below $x$ in column $c+1$ must be precisely $n$, then the $x$ on the right must lie in a row weakly higher than the $x$ of the left (otherwise this would imply that column $c$ has more than $n$ boxes). Thus, $\tilde{U}_{\alpha}$ has no attacking cells. 

To see $ \tilde{U}_{\alpha}$ has no co-inversion triples consider three cells of Type I with entries $i,j,k$ as in the left image in Figure \ref{fig:inv}. Suppose that $i<k$. In order for these cells to form a co-inversion triple then necessarily $i < j <k$. However, this would imply that there is a cell between $i$ and $k$ in the left column with value $j$, which is impossible since we showed that $\tilde{U}_{\alpha}$ cannot have attacking cells. If instead $i>k$ then the cells form a co-inversion triple if either $k>j$ or $j>i$. However, both of these cases result in cells that attack $j$. Hence $\tilde{U}_{\alpha}$ has no co-inversion triples of Type II. Now consider a triple of cells of Type II with entries $i,j,k$ as in the right image of Figure \ref{fig:inv}. Once again, the cells form a co-inversion triple for $k<j$ only if $k<i<j$ and for $k<j$ only if $i<j$ or $k<i$. However, as in the cases above, by the definition of $\tilde{U}_{\alpha}$ these cases give rise to attacking cells. Hence, $\tilde{U}_{\alpha}$ has no attacking cells of Type II and is indeed a semistandard key tabloid of shape $\alpha$. 
\end{proof}

It is easy to see $\tilde{U}_{\alpha}$ is a highest weight element in $\SSKD(\alpha)$ for any weak composition $\alpha$. Indeed, $\tilde{U}_{\alpha}$ is the unique such filling that satisfies this property for all $\alpha$. That is, if there exists a word $w = w_1 \dots w_k$ such that for all weak compositions $\alpha$ its corresponding tabloid of shape $\alpha$, $U'_{\alpha}$, is always contained in $\SSKD(\alpha)$ and satisfies $e_i(U'_{\alpha})=0$ for all $1\leq i \leq n$, then $U'_{\alpha}=\tilde{U}_{\alpha}$. Moreover, under the bijections described in Theorem~\ref{thm:isomorphism}, we have $\theta_{\alpha}(\tilde{U}_{\alpha}) = \tilde{u}_{\eta_{n,k}}$.

%
\section{Characters}
%
\label{sec:characters}

Recall the specialized nonsymmetric Macdonald polynomials $E_{\alpha}(x_1,\ldots,x_n;q,0)$ include the parameter $q$. Combinatorially, this parameter is given by the major index statistic defined as follows. For a semi-standard key tabloid $T$, the \newword{major index} of $T$, denoted by $\maj(T)$, is the sum of the legs of all cells $c$ such that the entry in $c$ is strictly less than the entry immediately to its right, as seen in Fig.~\ref{fig:maj}.

\begin{figure}[ht]
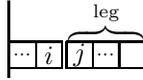

  \begin{displaymath}
    \begin{array}{c}
      \vline\tableau{_{\cdots}}\hss\tableau{i} \hspace{-.5mm} \overbrace{\tableau{j}\hss\tableau{_{\cdots}}\hss\tableau{ \ }}^{\mathrm{leg}} 
    \end{array}
  \end{displaymath}
  \caption{\label{fig:maj}The leg of a cell contributing to the major index, where $i<j$.}
\end{figure}

\begin{theorem}[\cite{HHL08}]
  The specialized nonsymmetric Macdonald polynomial is 
  \begin{equation}
    E_{\alpha}(x_1,\ldots,x_n;q,0) = \sum_{\substack{T \in \SSKD(\alpha)}} q^{\maj(T)} x_1^{\wt(T)_1} \cdots x_n^{\wt(T)_n} .
  \end{equation}
  \label{thm:mac}
\end{theorem}

In \cite{AG21}, we show the major index is constant on connected components of the \emph{finite} Demazure crystal on semistandard key tabloids. By Theorem~\ref{thm:isomorphism}, the affine Demazure crystal on semistandard key tabloids is connected, and so we wish to understand the role of the parameter $q$ in this context. 

An \emph{energy function} on a crystal $\mathcal{B}$ is a function $H: \mathcal{B} \otimes \mathcal{B} \to \mathbb{Z}$ satisfying the following conditions for all $0\leq i <n$ and $b_1 \otimes b_2 \in \mathcal{B} \otimes \mathcal{B}$ such that $e_i(b_1\otimes b_2) \neq 0$. 
\begin{equation}
H(e_i(b_1 \otimes b_2)) = \begin{cases}
H(b_1 \otimes b_2) & i\neq 0 \\
H(b_1 \otimes b_2) +1  & i=0, \varphi_0(b_1) \geq \varepsilon_0(b_2) \\
H(b_1 \otimes b_2) -1  & i=0, \varphi_0(b_1) < \varepsilon_0(b_2).
\end{cases}
\end{equation}
Given two crystals $\B,\B'$, there is a unique map $R$, called the \emph{combinatorial $R$-matrix}, from $\B\otimes\B'$ to $\B'\otimes\B$ that intertwines the crystal operators \cite{Kas02}. Using the combinatorial $R$-matrix and a local energy function $H$, we may consider the \emph{global energy function} $E:\mathcal{B}^{\otimes n} \to \mathbb{Z}$ such that, 
\begin{equation}
E(b_1 \otimes \dots \otimes b_n) = \sum_{i=1}^{n-1} (n-i) \cdot H(b_{i} \otimes b_{i+1}).
\end{equation}
This produces an important grading on finite dimension modules that connects to the major index statistic above as follows. 

\begin{proposition}\label{prop:lowering-energy}
 The major index $\maj: \SSKD(\alpha) \to \mathbb{Z}$ is a global energy function.
\end{proposition}

\begin{proof}
  Let $\alpha$ be a weak composition consisting of one row of length $n$. If $n=2$ it is straight forward to see that $maj$ will correspond to an energy function for $\SSKD(\alpha)$. Now suppose that $n>2$ and for any $T \in \SSKD(\alpha)$ let $b_i$ be the entry in the $i^{th}$ column of $T$. Then, 
  \begin{align*}
    \maj\left(\hackcenter{\begin{ytableau} b_1 & \dots & b_n\end{ytableau}}\right) & = \sum_{i=1}^{n-1} \text{leg}(b_{i})\cdot \maj\left(\boxed{b_{i_{~}}}\boxed{b_{i+1}}\right)\\
    &= \sum_{i=1}^{n-1}(n-i)\cdot \maj\left(\boxed{b_{i_{~}}}\boxed{b_{i+1}}\right)\\
    &= \sum_{i=1}^{n-1} (n-i) \cdot H(b_i \otimes b_{i+1})\\
    &=E(b_1 \otimes \dots \otimes b_n).
  \end{align*}
  Since for any multi-row weak composition $\alpha$, the major index of any $T \in \SSKD(\alpha)$ is the sum of the major index along each row of $T$, then the general result follows immediately. 
\end{proof}

Nakayashiki and Yamada \cite{NY97} first made the connection between Hall--Littlewood polynomials, given by $P_{\lambda}(X;0,t)$, and $\asl$ by showing the celebrated \emph{charge statistic} \cite{LS78,But86,But94} is (up to rescaling) an energy function in solvable lattice models. For details relating Hall--Littlewood polynomials with the nonsymmetric Macdonald polynomials specialized at $t=0$, see \cite[Cor.~5.7]{Ass18}.

In \cite{San00}, Sanderson proves that specialized nonsymmetric Macdonald polynomials are characters of affine Demazure modules by defining a family of operators $H_i$ for $0\leq i <n$ on $\mathbb{Z}[q,q^{-1}][X]$, previously introduced by Knop \cite{Kno97} and Sahi \cite{Sah96}, which generate these polynomials and satisfy the relation $H_iE_{\alpha}(X;q,0)= E_{s_i\cdot \alpha}(X;q,0)$ for each $i$ (c.f. \cite[Theorem~1]{San00}). Our embedding operators are precisely crystal theoretic lifts of Sanderson's operators, thus an immediate consequence of our work is a new proof of Sanderson's result \cite[Theorem~6]{San00}.

\begin{corollary}
  Given $\alpha = w \cdot \eta_{n,k}$, we have $E_{\alpha}(x_1,\ldots,x_n;q,0) = \mathrm{ch}(\aB_w(\eta_{n,k}))$.
  \label{cor:sanderson}
\end{corollary}



As a final application, by forgetting the $0$-edges in the affine Demazure crystal on semistandard key tabloids, we recover the finite Demazure crystal on semistandard key tabloids from \cite{AG21}, and so we can also now interpret \cite[Theorem~4.9]{Ass18} in terms of characters.

\begin{corollary}
  The affine Demazure characters decompose as $q$-graded sums of finite Demazure characters.
\end{corollary}

For example, taking the $q$-graded character of the Demazure crystals in Fig.~\ref{fig:complete example}, the characters of 
$\tilde{\mathcal{B}}_{s_0s_2}(2,2,1)$, $\tilde{\mathcal{B}}_{s_1s_0s_2}(2,2,1)$, and $\tilde{\mathcal{B}}_{s_0s_1s_0s_2}(2,2,1)$ are precisely
\begin{align*}
E_{(3,1,1)}(X;q,0) &= q\kappa_{(2,1,2)}(X)+\kappa_{(3,1,1)}(X),\\
E_{(1,3,1)}(X;q,0) &= q\kappa_{(1,2,2)}(X)+\kappa_{(1,3,1)}(X),\\
E_{(2,3,0)}(X;q,0) &= q^2\kappa_{(1,2,2)}(X)+q\kappa_{(1,3,1)}(X)+q\kappa_{(2,2,1)}(X)+\kappa_{(2,3,0)}(X).
\end{align*}
In particular, the affine crystal operator $f_0$ connects the finite Demazure subcrytals of each affine Demazure crystal and, unlike the finite crystal operators, does not preserve the major index. Similarly, only the \emph{affine} embedding operator $\D_0$ changes the $q$-grading of each of the finite Demazure subcrystals, increasing it by a factor of $q$ each time. At the level of characters, we see that $\D_1$ and $\D_0$ recover the action of Sanderson's operators $H_1$ and $H_0$ on the specialized nonsymmetric polynomials.

Recall the Schur polynomials form an important basis for symmetric polynomials whose structure constants give the multiplicities of the irreducible components in the tensor product of irreducible representations. The structure constants for the Hall--Littlewood symmetric polynomials, on the other hand, are not nonnegative (see \cite[(III.3)]{Mac95}). We can understand this failure through the crystal interpretation of $E_{(\lambda_n,\ldots,\lambda_1)}(x_1,\ldots,x_n;q,0)$ as an affine \emph{Demazure} character since the tensor product of (affine) Demazure crystals is not, in general, a(n affine) Demazure crystal. Nevertheless, this perspective might lead to a better understanding of these Hall--Littlewood structure constants as tensor products are well-defined on crystals, even when the resulting structure is not well understood.

%
%

\bibliographystyle{amsplain} 
\bibliography{affine}


\end{document}